\documentclass[11pt, a4paper]{article}
\title{\bf Nowhere differentiable functions of analytic type on products of 
finitely connected planar domains}

\date{}
\usepackage[T1]{fontenc}
\usepackage[greek,english]{babel}  
\usepackage[utf8x]{inputenc}
\usepackage{hyperref}
\RequirePackage{latexsym} 
\usepackage{amsmath,amssymb,amsthm}
\usepackage{a4wide}
\usepackage{authblk}
\newcommand{\norm}[1]{\left\lVert#1\right\rVert}
\usepackage{enumerate}

\usepackage{dsfont}
\usepackage{mathrsfs}
\usepackage[margin= 3.5cm]{geometry}

\newtheorem{theorem}{Theorem}[section]
\newtheorem{lemma}[theorem]{Lemma}
\newtheorem{proposition}[theorem]{Proposition}

\newtheorem{definition}[theorem]{Definition}

\newtheorem{remark}[theorem]{Remark}

\numberwithin{equation}{section}

\begin{document}
\author[1]{Vlassis Mastrantonis} \author[2]{Christoforos Panagiotis\thanks{Supported by the European Research Council (ERC) under the European Union's Horizon 2020 research and innovation programme (grant agreement No 639046). Research was conducted at the University of Athens, Department of Mathematics}}

\affil[1]{Department of Mathematics, University of Athens, Panepistimioupolis, 157 84, Athens, Greece.\\
E-mail: \href{mailto:blasismas@hotmail.com}{\url{blasismas@hotmail.com}}}

\affil[2]{{Mathematics Institute}\\
 {University of Warwick}\\
  {CV4 7AL, UK}\\
E-mail: \href{mailto:C.Panagiotis@warwick.ac.uk}{\url{C.Panagiotis@warwick.ac.uk}}}
\maketitle

\vspace{-6ex}
\begin{abstract}
\footnotesize
Using Laurent decomposition and Mergelyan's theorem combined with Baire's category theorem, we prove generic nowhere differentiability on the distinguished boundary of functions of analytic type on products of planar domains bounded by finitely many disjoint Jordan curves. The parametrization of the boundaries are those induced by the natural Riemann maps.
\end{abstract}

\section{\textbf{Introduction}}

In the nineteenth century, it was widely believed that every continuous function defined on 
$\mathbb{R}$ has derivative at a significant set of points. Weierstrass was the first one to 
disprove this claim by providing an explicit example of a function $u:\mathbb{R}\rightarrow
\mathbb{R}$ which is continuous, 2$\pi$-periodic, but not differentiable at any real number. 
Since then the existence of such functions attracted the interest of many mathematicians.

In fact, as Banach and Mazurkiewicz independently proved, in \cite{[1]} and \cite{[6]} respectively, 
generically every continuous function on a compact interval $J$ of $\mathbb{R}$ is
nowhere differentiable. In \cite{[2]}, the above result has been generalized for
functions in the disc algebra $A(D)$, i.e. continuous functions defined on
the closed unit disc $\overline{D}$, which are also holomorphic on $D$. In particular, it is 
proved that generically for every function $f\in A(D)$ both functions 
$Ref\restriction_\mathbb{T}$ and $ Imf\restriction_\mathbb{T}$ are nowhere 
differentiable. The result is even extended for functions 
$f : \overline{D}^I\rightarrow \mathbb{C}$ that belong to the algebra $A(D^I)$, namely $f$ is continuous on $\overline{D}^I$, which is endowed with the cartesian topology, and separately
holomorphic on $D^I$.

In section 4, we prove a similar result in the case of a domain $\Omega$ bounded by a 
finite number of disjoint Jordan curves. If $\Omega_0, \Omega_1,...\Omega_{n}$ are the 
complements of the connected components of $(\mathbb{C}\cup \{\infty\})\setminus\Omega$, then the Riemann mappings 
$\phi_i: \overline{D}\rightarrow \overline{\Omega_i}$, $i=1,...,n$ yield 
a parametrization of $\partial \Omega_i$, $i=0,1,...,n$, with parameter $\theta\in\mathbb{R}$. We will prove that generically for 
every $f\in A(\Omega)$ the functions $Re(f\circ\phi_i)\restriction_\mathbb{T}$ and $
Im(f\circ\phi_i)\restriction_\mathbb{T}$ are nowhere differentiable with respect to 
$\theta\in\mathbb{R}$. In order to do so, we combine the Laurent decomposition of every function $f\in A(\Omega)$ with the results of the previous sections which deal with the case of the disc or more generally with the case of Jordan domains. 

In section 5, we generalize the above result in the context of several complex variables. Specifically, if we consider sets 
$\Omega$ which are the finite or countably infinite cartesian products of domains in $\mathbb{C}$ as in section 4, then a parametrization of the distinguish boundary of $\Omega$ arises from the Riemann mappings of the factors $\Omega_i$ of the product $\Omega$. 
We prove that, generically, for every $f\in A(\Omega)$ neither
the real part $Ref$, nor the imaginary part 
$Imf$ have directional
derivatives at any point of the distinguished boundary of $\Omega = \prod_{i\in I}\Omega_i$
with respect to the above parametrization, for any direction
$v\in l^\infty\setminus\{0\}$. We first treat the case of finite products that is when the set $I$ is finite. In the case of infinitely many variables, we use the result for the case of finitely many variables and the fact that every $f\in A(\Omega)$ can be uniformly approximated by functions depending on a finite set of variables (\cite{[3]}, \cite{[5]}).

\section{The open unit disk}

We denote the open unit disk in $\mathbb{C}$ by $D$ and $\mathbb{T}=\partial D$. Let 
$\Omega$ be a domain in $\mathbb{C}$. Although the only domains we will deal with in this 
sections are $D$ and $\hat{\mathbb{C}}\setminus \overline{D}$ 
we will give a general definition which will be useful more generally.

\begin{definition}
Let $\Omega$ be a domain in $\mathbb{C}$. A function 
$f:\overline{\Omega} \rightarrow \mathbb{C}$ belongs to $A(\Omega)$ if it is continuous on
$\overline{\Omega}$ and holomorphic on $\Omega$. We 
endow $A(\Omega)$ with the topology of uniform convergence on compact subsets of $\overline{\Omega}$, where the closure of $\Omega$ is taken in $\mathbb{C}$.
\end{definition}

We note that a function defined on $\mathbb{T}$ can equivalently be thought as a 2$\pi$-periodic function defined on $\mathbb{R}$. Thus, by abuse of notation we write $u(y)$ instead 
of $u(e^{iy})$ for $y\in\mathbb{R}$.

\begin{definition}
A function $f\in C(\mathbb{T})$ belongs to $Z$ if for every $\theta\in\mathbb{R}$
\begin{equation}
\limsup\limits_{y\rightarrow\theta^+} \left| \frac{Ref(y) - Ref(\theta)}{y-\theta} 
\right| = +\infty
\end{equation}
and
\begin{equation}
\limsup\limits_{y\rightarrow\theta^+} \left| \frac{Imf(y) - Imf(\theta)}{y-\theta} 
\right| = +\infty  
\end{equation}
\end{definition}

In \cite{[2]}, it has been proven that the set of functions $f\in A(D)$ such that 
$f\restriction{_\mathbb{T}}$ belongs to $Z$, which we denote by $Z(D)$, 
is a dense and $G_{\delta}$ subset of the disk algebra $A(D)$. 
 
It will be useful to restrict ourselves to functions $f\in A(D)$, such that $f(0)=0$. We 
denote the space of such functions by $A_0(D)$ and we endow it with the subspace topology. Using the above result we are going to prove 
that the functions $f\in A_0(D)$ such that $f\restriction_{\mathbb{T}}\in Z$ is dense and $G_\delta$ in $A_0(D)$. Denote this set by $Z_0(D)$.

\begin{proposition}
The set of functions $f\in A_0(D)$ such that $f\restriction{_\mathbb{T}}$ belongs to $Z$ is a 
dense and $G_{\delta}$ subset of $A_0(D)$.
\end{proposition}

\begin{proof}
Observe that $Z_0(D)= Z(D)\cap A_0(D)$.
Consider the mapping
\begin{align*}
L: A_0(D)& \longrightarrow A(D)\\
f& \longmapsto f.
\end{align*}
$L$ is obviously continuous and it is easy to see that 
$L^{-1}(Z(D))=Z_0(D)$ is $G_\delta$. 
So it suffices to show that $Z_0$ is dense in $A_0(D)$. 

Let $\epsilon>0$ and 
$f\in A_0(D)\subset A(D)$. Since $Z(D)$ is dense in $A(D)$ there exists a 
$g\in Z(D)$ such that 
\begin{equation}
\norm{f-g}_{\infty}<\frac{\epsilon}{2}.
\end{equation}
More specifically 
\begin{equation}
|f(0)-g(0)|<\frac{\epsilon}{2}
\end{equation}
It is now obvious that $\tilde{g}(z)=g(z)-g(0)\in A_0(D)$ and $\tilde{g}\in Z_0(D)$. 
Using the triangular inequality we deduce that 
\begin{equation}
\norm{f-\tilde{g}}_{\infty}<\epsilon,
\end{equation}
which completes the proof.
\end{proof}

We will now prove similar results for a subspace of $A(\hat{\mathbb{C}}\setminus\overline{D})$ denoted by
$A_0(\hat{\mathbb{C}}\setminus\overline{D})$. In general, if the domain $\Omega$ is the complement of a compact set, 
then a function $f$ belongs to $A_0(\Omega)$ if 
$f\in A(\Omega)$ and 
\begin{equation}
\lim_{z\to\infty}f(z)=0.
\end{equation}

\begin{proposition}\label{exterior}
The set of functions $f\in A_{0}(\hat{\mathbb{C}} \setminus \overline{D})$ such that 
$f\restriction_{\mathbb{T}}\in Z$ is a dense and $G_\delta$ subset of 
$A_{0}(\hat{\mathbb{C}} \setminus \overline{D})$.
\end{proposition}
\begin{proof}
Consider the mappings
\begin{align}
\phi: \bar{D} &\longrightarrow \hat{\mathbb{C}} \setminus D \\
\nonumber z & \longmapsto \frac{1}{z}
\end{align}
and
\begin{align}
l: A_{0}(\hat{\mathbb{C}} \setminus \bar{D})& \longrightarrow A_{0}(D)\\
\nonumber f& \longmapsto f\circ \phi
\end{align}
We can easily see that $l$ is continuous, onto $A_{0}(D)$ and 
\begin{equation}
\norm{l(f) - l(g)}_{\infty} = \norm{f-g}_{\infty}.
\end{equation}
Hence, $l$ is an isometric isomorphism, which implies that the set
$Z_0(\hat{\mathbb{C}} \setminus \overline{D}):=l^{-1}(Z_0(D))$ is dense and $G_{\delta}$ in 
$A_{0}(\hat{\mathbb{C}} \setminus \bar{D})$.

We will now prove that $Z_0(\hat{\mathbb{C}} \setminus \overline{D})$ coincides with the set of 
functions $f\in A_{0}(\hat{\mathbb{C}} \setminus \overline{D})$ such that 
$f\restriction_{\mathbb{T}}\in Z$. This
derives from the following equalities:

\begin{equation}
\left| \frac{u(y) - u(\theta)}{y-\theta} \right| = \left| \frac{(u\circ\phi)(-y) - (u\circ\phi)
(-\theta)}{(-y)-(-\theta)} \right| ,
\end{equation}

\begin{equation}
\left| \frac{v(y) - v(\theta)}{y-\theta} \right| = \left| \frac{(v\circ\phi)(-y) - (v\circ\phi)
(-\theta)}{(-y)-(-\theta)} \right|,
\end{equation}
where $u,v$ are the real and imaginary part of a function $f\in A_{0}(D)$ respectively and 
$\theta,y\in\mathbb{R}$ are such that $y\neq\theta$.
\end{proof}

\section{Jordan Domains}

In this section we replace the unit circle $\mathbb{T}$ by a Jordan curve $J$ and we prove similar results for functions defined on the domain $\Omega$, which is the bounded 
or unbounded connected component of the complement in $\hat{\mathbb{C}}$ of $J$. 
For such domains $\Omega$, we consider a Riemann mapping 
$\phi: \overline{D}\rightarrow\overline{\Omega}$ and we claim that generically for 
every function $f\in A(\Omega)$, the function $(f\circ \phi)\restriction_{\mathbb{T}}$ belongs
to $Z$. In what follows, we prove the aforementioned assertions.

\begin{theorem}\label{Jordan}
Let $\Omega$ be a Jordan domain and $\phi: \overline{D}\rightarrow\overline{\Omega}$ a Riemann 
mapping. The set of functions $f\in A(\Omega)$ such that 
$(f\circ \phi)\restriction_{\mathbb{T}}$ belongs to $Z$ is a dense and $G_{\delta}$ subset of 
$A(\Omega)$.
\end{theorem}
\begin{proof}
Let
\begin{align}
L: A(\Omega)& \longrightarrow A(D)\\
\nonumber f& \longmapsto f\circ \phi
\end{align}
It is easy to verify that $L$ is an isometric isomorphism, which implies that the set 
$L^{-1}(Z(D))$ is a dense and $G_{\delta}$ subset of 
$A(\Omega)$.

Furthermore $L^{-1}(Z(D))$ coincides with the set of functions $f\in A(\Omega)$ such that 
$(f\circ\phi)\restriction_{\mathbb{T}}$ belongs to $Z$, because, for any $f\in A(\Omega)$ we have that:
$(f\circ\phi)\restriction_{\mathbb{T}}\in Z \Leftrightarrow L(f)\restriction_{\mathbb{T}}\in Z \Leftrightarrow L(f)\in Z(D) \Leftrightarrow f\in L^{-1}(Z(D))$.
\end{proof}

\begin{theorem}\label{complement}
Let $\gamma$ be a Jordan curve and $\Omega$ the unbounded connected component of 
$\mathbb{C}\setminus\gamma^*$. Let also 
$\phi: \overline{D}\rightarrow\overline{\Omega\cup\{\infty\}}$ be a Riemann 
mapping. The set of functions $f\in A(\Omega\cup\{\infty\})$ such that 
$(f\circ \phi)\restriction_{\mathbb{T}}$ belongs to $Z$ is a dense and $G_{\delta}$ subset of 
$A(\Omega\cup\{\infty\})$.
\end{theorem}
\begin{proof}
Consider the mapping
\begin{align}
\mathcal{L}: A_{0}(\hat{\mathbb{C}}\setminus\bar{\Omega})& \longrightarrow A_0(\hat{\mathbb{C}}
\setminus \bar{D})\\
\nonumber f& \longmapsto f\circ \phi.
\end{align}
Similarly to Proposition \ref{exterior} the desired set is a dense and $G_{\delta}$ subset of 
$A(\Omega\cup\{\infty\})$.
\end{proof}

\begin{remark} In Theorems \ref{Jordan} and \ref{complement} every Riemann mapping $\Phi: D\rightarrow
\Omega$ can be obtained from the composition of a fixed Riemann mapping $\phi$ of $\Omega$ 
with a Möbius transformation mapping $D$ onto itself. Therefore, one can easily see that if the functions 
$Re{(f\circ\phi)}{\restriction_{\mathbb{T}}}$ and $Im{(f\circ\phi)}{\restriction_{\mathbb{T}}}$
are nowhere differentiable with respect to $\theta\in \mathbb{R}$, then the same holds true for 
the functions $Re{(f\circ\Phi)}{\restriction_{\mathbb{T}}}$ and 
$Im{(f\circ\Phi)}{\restriction_{\mathbb{T}}}$. The above implication may not be true for an 
arbitrary parametrization of $\partial \Omega$.
\end{remark}

\section{Domains bounded by a finite number of disjoint Jordan curves}

Now we consider a more general setting, where $\Omega$ is a bounded domain whose boundary consists of a finite number of disjoint Jordan curves. If
$V_0,...,V_{n-1}$ are the connected components of $\hat{\mathbb{C}}\setminus\Omega$, $\infty \in V_0$
and $\Omega_0=\hat{\mathbb{C}}\setminus V_0,...,\Omega_{n-1}=\hat{\mathbb{C}}
\setminus V_{n-1}$, then there exist Riemann mappings $\phi_i:\overline{D}\rightarrow \overline{\Omega_i}$. We 
are interested in functions $f\in A(\Omega)$, such that 
$({f\circ\phi_i}){\restriction_{\mathbb{T}}}$ belongs to $Z$, for all $i=0,...,n-1$. 

For every $k\in\mathbb{N}$ and $i\in\{0,1,...,n-1\}$ consider the sets 
\begin{equation}
\begin{split}
D_k=\{u\in C_{\mathbb{R}}(T):\text{for every}\ \theta\in\mathbb{R} \text{ there exists}\ 
y\in\left(\theta,\theta+\dfrac{1}{k}\right)\\
\text{ such that } |u(y)-u(\theta)|>k|y-\theta|\}
\end{split}
\end{equation}
and
\begin{equation}
E^{(i)}_k=\{f\in A(\Omega): Re(f\circ\phi_i)\restriction_{\mathbb{T}}\in D_k\}
\end{equation}

\begin{lemma}\label{open}
For every $k\in\mathbb{N}$ and $i\in\{0,1,...,n-1\}$ the set $E^{(i)}_k$ is an open subset of 
$A(\Omega)$.
\end{lemma}

\begin{proof}
Let $k\in\mathbb{N}$ and $i\in\{0,1,...,n-1\}$. We will prove equivalently that $A(\Omega)
\setminus E_k^{(i)}$ is a 
closed subset of $A(\Omega)$. Let ${f_m}$ be a sequence in $A(\Omega)\setminus 
E^{(i)}_k$ converging uniformly on $\overline{\Omega}$ to a function $f\in A(\Omega)$.
For every $m\in\mathbb{N}$ there exist $\theta_m\in\mathbb{R}$ such that 
\begin{equation} 
\left|\dfrac{u_m(\phi_i(y))-u_m(\phi_i(\theta_m))}{y-\theta_m}\right|\leq k,
\end{equation}
where $u_m=Ref_m$.
Since every $u_m$ is 2$\pi$- periodic we may assume that every $\theta_m \in[0,2\pi]$ and thus 
the sequence $\{\theta_m\}$ has a convergent subsequence. Without loss 
of generality we assume that $\{\theta_m\}$ converges to a $\theta\in[0,2\pi]$. 

Let $y\in\left(\theta,\theta+\dfrac{1}{k}\right)$. We can easily deduce that there exists 
$m_0\in\mathbb{N}$, such that for every $m\geq m_0$ it is true that 
$y\in\left(\theta_m,\theta_m+\dfrac{1}{k}\right)$. Thus, applying $(4.3)$ and using the 
convergence of $\{\theta_m\}$ to $\theta$ and the uniform convergence of $\{f_m\}$ to $f$ we 
have that
\begin{equation} 
\left|\dfrac{u(\phi_i(y))-u(\phi_i(\theta))}{y-\theta}\right|\leq k,
\end{equation}
where $u=Ref$.
Hence $(4.4)$ holds for every $y\in\left(\theta,\theta+\dfrac{1}{k}\right)$ and thus 
$f\in A(\Omega)\setminus E_k^{(i)}$. Therefore, $A(\Omega)\setminus E^{(i)}_k$ is a closed subset 
of $A(\Omega)$.
\end{proof}

\begin{theorem}
The class
\begin{equation}
S=\bigcap_{i=0}^{n-1}\bigcap_{k=1}^\infty E^{(i)}_k
\end{equation}
is a dense and $G_\delta$ subset of $A(\Omega)$.
\end{theorem}

\begin{proof}
Lemma \ref{open} indicates that the sets
\begin{equation}
S_i=\bigcap_{k=1}^\infty E^{(i)}_k
\end{equation}
are $G_\delta$. We will prove that $S_0$
is dense in $A(\Omega)$. 

Let $f\in A(\Omega)$ and $f_i\in A(\Omega_i)$ be such that 
$f=f_0+f_1+...+f_{n-1}$. Let also $g=f_1+...+f_{n-1}$.
Since $\overline{\Omega}\subset\overline{\Omega_0}$ is compact, 
$\hat{\mathbb{C}}\setminus\Omega$ has $n$ connected components and $\phi_0^{-1}:
\overline{\Omega_0}\rightarrow\mathbb{C}$ 
is a conformal map, the set $\phi_0^{-1}(\overline{\Omega})\subset \overline{D}\subset\mathbb{C}$ is compact and 
$\hat{\mathbb{C}}\setminus\phi_0^{-1}(\overline{\Omega})$ has $n$ connected components.
The function $g\circ\phi_0: \phi_0^{-1}(\overline{\Omega})\rightarrow\mathbb{C}$ belongs to 
$A(\phi_0^{-1}(\Omega))$ and by Mergelyan's theorem there exists a sequence $\{r_m\}$ of 
rationals functions with poles off $\phi_0^{-1}(\overline{\Omega})$ which converges to 
$g\circ\phi_0$ uniformly on 
$\phi_0^{-1}(\overline{\Omega})$. Thus, the sequence $\{r_m\circ\phi_0^{-1}\}$ converges to 
$g$ uniformly on $\overline{\Omega}$. 

Since $f_0\in A(\Omega_0)$, there exists a sequence $\{g_m\}$, such that $g_m\in A(\Omega_0)$ and $g_m\circ\phi_0\in Z(D)$, which converges to $f_0$. From the above we have that the 
sequence $\{g_m+r_m\circ\phi^{-1}_0\}$ converges uniformly to $f$ in $A(\Omega)$ and also that:
\begin{gather}
\left|\dfrac{Re(g_m+r_m\circ\phi_0^{-1})(\phi_0(y))-Re(g_m+r_m\circ\phi_0^{-1})
(\phi_0(\theta))}{y-\theta}\right|\geq \\
\nonumber \left|\dfrac{Reg_m(\phi_0(y))-Reg_m(\phi_0(\theta))}{y-\theta}\right|-
\left|\dfrac{Re(r_m\circ\phi_0^{-1})(\phi_0(y))-Re(r_m\circ\phi_0^{-1})
(\phi_0(\theta))}{y-\theta}\right|= \\
\nonumber \left|\dfrac{Reg_m(\phi_0(y))-Reg_m(\phi_0(\theta))}{y-\theta}\right|-
\left|\dfrac{Rer_m(y)-Rer_m(\theta)}{y-\theta}\right|
\end{gather}
for every $\theta,y\in\mathbb{R}$, $\theta\neq y$.
From $(4.7)$ and the fact that the functions $r_m$ are differentiable at the unit circle, 
every $g_m+r_m\circ\phi_0^{-1}$ belongs to $S_0$. Therefore, the set $S_0$ is dense.
 
Similarly, the sets $S_1,S_2,...,S_{n-1}$ are dense and Baire's theorem completes the proof.
\end{proof}

\begin{theorem}\label{curves}
The class of functions $f\in A(\Omega)$ such that for every $i\in\{0,1,...,n-1\}$, the 
functions $(Ref\circ\phi_i){\restriction}_{\mathbb{T}}$ and 
$(Imf\circ\phi_i){\restriction}_{\mathbb{T}}$ are nowhere differentiable is residual in 
$A(\Omega)$.
\end{theorem}

\begin{proof}
It is easy to see that $S\cap iS$ is a subset of all functions $f\in A(\Omega)$ such that for 
every $i\in\{0,1,...,n-1\}$ the functions $(Ref\circ\phi_i){\restriction}_{\mathbb{T}}$ and $
(Imf\circ\phi_i){\restriction}_{\mathbb{T}}$ have no derivatives at any point. Since $S$ is $G_{\delta}$ and 
dense in $A(\Omega)$ and multiplication by $i$ is an automorphism of $A(\Omega)$, we conclude 
that $iS$ is $G_{\delta}$ and dense in $A(\Omega)$. Applying Baire's category theorem we have 
that $S\cap iS$ is $G_{\delta}$ dense in $A(\Omega)$.
\end{proof}

\section{The case of several variables}

Let $I$ be a non empty finite or countably infinite set and $\{\Omega_i\}_{i\in I}$ be a family 
of domains in $\mathbb{C}$, all of them bounded by a finite number of disjoint Jordan curves as in section 4. 
We denote by $k_i$ the degree of connectivity of $\Omega_i$. In accordance 
to the notation of the previous sections, we denote by $V_{i,j}$, $i\in I$,
$j\in\{0,...,k_i-1\}$ the connected components of 
$\hat{\mathbb{C}}\setminus\overline{\Omega_i}$ 
and $\Omega_{i,j}=\hat{\mathbb{C}}\setminus \overline{V_{i,j}}$. Let also 
$\infty\in \Omega_{i,0}$ and $\phi_{i,j}:\overline{D}\rightarrow\overline{\Omega_{i,j}}$ be
some fixed Riemann maps. 

We consider the spaces $\Omega=\prod_{i\in I}\Omega_i$, 
$\overline{\Omega}=\prod_{i\in I}\overline{\Omega_i}$ and $b_0 \Omega=\prod_{i\in
 I}\partial\Omega_i$, the distinguished boundary of $\Omega$, all of them endowed with the 
product topology. According to \cite{[3]}, \cite{[4]} and \cite{[5]} we give the following definition.

\begin{definition} A function $f:\overline{\Omega}\rightarrow\mathbb{C}$ belongs to the space 
$A(\Omega)$ if it is continuous
on $\overline{\Omega}$, endowed with the product topology, and separately holomorphic on $\Omega$. By
separately holomorphic we mean that if all coordinates but one are fixed, say the
coordinate $i_0$, then the restriction of $f$ is a holomorphic function of the variable 
$z_{i_0}$ in $\Omega_{i_0}$. We endow the space $A(\Omega)$ with the supremum norm.
\end{definition}

We also consider the set
\begin{equation}
R=\prod_{i\in I}\{0,...,k_i-1\}
\end{equation} 
and the functions 

\begin{align}
\phi_r: \overline{D}^I &\longrightarrow \overline{\Omega_{1,r_1}} \times 
\overline{\Omega_{2,r_2}} \times ... \quad = \prod_{i\in I}\overline{\Omega}_{i,r_i} \\ \nonumber
(z_i)_{i\in I}&\longmapsto(\phi_{i,r_i}(z_i))_{i\in I}
\end{align}
where $r=(r_i)_{i\in I}\in R.$ 

We are going to prove that, generically, for every $f\in A(\Omega)$ the functions $Re(f\circ \phi_r)\restriction_{\mathbb{T}^I}$ and $Im(f\circ\phi_r)\restriction_{\mathbb{T}^I}$ have no directional 
derivatives for any $r\in R$ and any direction $v\in \l^\infty(I)\setminus\{0\}$. 
In order to do so, we will first treat the case where $I$ is a finite set and separately the case 
where $I$ is a countably infinite set.

\subsection{Finitely many variables}

Let $I$ be a non empty finite set.

We consider the sets
\begin{gather}
 D_{n,k,s}(I)=\Big\{u\in C_{\mathbb{R}}(\mathbb{T}^I):\text{for every } \theta\in\mathbb{R}^
I \text{ and every direction }\\ v\in \mathbb{R}^I 
\nonumber \text{ with } \frac1s\leq|v_k|\leq s \text{ and } |v_j|\leq s \text{ for every } j\in I  \text{ there 
exists } \\
y\in\left(\theta-\dfrac{1}{n}v,\theta+\dfrac{1}{n}v\right) \text{ such that } 
\nonumber |u(\theta)-u(y)|>n\Vert y-\theta \Vert_{\infty}\Big\}
\end{gather}
for some $n,s\in\mathbb{N}$ and $k\in I$. Also consider:
\begin{equation}
E_{n,k,s}^{(r)}(I)= \{f\in A(\Omega): Re(f\circ\phi_r)\restriction_{\mathbb{T}^I}\in D_{n,k,s}\}.
\end{equation}
for some $r\in R$.

\begin{lemma}\label{many}
For every $n\in I$ and $r\in R$ the set $E_{n,k,s}^{(r)}(I)$ is an open subset of $A(\Omega)$.
\end{lemma}

\begin{proof}
We will prove that the set $A(\Omega)\setminus E_{n,k,s}^{(r)}(I)$ is closed. We consider a sequence
$\{f_m\}$ such that every $f_m\in A(\Omega)\setminus E_{n,k,s}^{(r)}(I)$, which uniformly converges 
to a function $f\in A(\Omega)$. For each $m$ there exists  $\theta^{(m)}\in \mathbb{T}^I$ and a 
direction $v^{(m)}\in\mathbb{R}^I$ with $|v^{(m)}_j|\leq s$ for all $j\in I$ and $\frac1s\leq |v^{(m)}_k|\leq s$, such that:
\begin{equation}
|(u_m\circ\phi_{r})(y)-(u_m\circ\phi_{r})(\theta^{(m)})|\leq n
\Vert{ y-\theta^{(m)}}\Vert_\infty
\end{equation}
for every 
$y\in\left(\theta^{(m)}-\dfrac{1}{n}v^{(m)},\theta^{(m)}+\dfrac{1}{n}v^{(m)}\right)$, where
$u_m=Ref_m$.
But $\mathbb{T}^I$ and the set 
$\{v\in\mathbb{R}^I: |v_j|\leq s, \forall j\in I \text{ and } \frac1s\leq|v_k|\leq s\}$ are 
metrizable and compact by Tychonoff's theorem. Thus, the sequences $\{\theta^{(m)}\}_{m\in\mathbb{N}}$ and 
$\{v^{(m)}\}_{m\in\mathbb{N}}$ have converging subsequences. Without loss of generality we assume that 
$\theta^{(m)}\rightarrow\theta\in\mathbb{R}^I$ and 
$v^{(m)}\rightarrow v\in\mathbb{R}^I$ with $|v_j|\leq s$ for all $j\in I$ and $\frac1s\leq |v_k|\leq s$.

If $y\in\left(\theta-\dfrac{1}{n}v,\theta+\dfrac{1}{n}v\right)$, then there exists 
$-\dfrac{1}{n}<t<\dfrac{1}{n}$ such that 
\begin{equation}
y=\theta+tv=\lim_{m\rightarrow\infty}(\theta^{(m)}+tv^{(m)}).
\end{equation}

Since $f_m$ converges uniformly to $f$, it follows that $u_m$ converge uniformly to $u$, where 
$u=Ref$ and $\phi_r$ is continuous which implies that $\phi_r(\theta^{(m)} +tv^{(m)}) 
\rightarrow \phi_r(\theta +tv)$, as $ m\to\infty$. The above combined with $(5.5)$ imply that

\begin{equation}
|(u\circ\phi_{r})(y)-(u\circ\phi_{r})(\theta)|\leq n
\Vert{ y-\theta}\Vert_\infty
\end{equation}
and therefore $f\in A(\Omega)\setminus E_{n,k,s}^{(r)}(I)$ which proves the lemma.
\end{proof}
\vspace{0.5cm}
Now let 
\begin{equation}
S_{k,s}^{(r)}(I) = \bigcap\limits_{n=1}^{\infty}E_{n,k,s}^{(r)}(I) 
\end{equation}
for all $r\in R$, $k\in I$ and $s\in\mathbb{N}$. Obviously, $S_{k,s}^{(r)}(I)$ is $G_\delta$ in $A(\Omega)$.

\begin{lemma}\label{dense}
For all $r\in R$ and $k,s\in I$ the class $S_{k,s}^{(r)}(I)$ is dense in $A(\Omega)$.
\end{lemma}

\begin{proof}
Let us show that $S_{k,s}^{(r)}(I)$ is non empty. We have proved in a previous section that there is a 
function $g\in A(\Omega_k)$ such that

\begin{equation}
\limsup\limits_{y\rightarrow\theta^+} \left| \frac{Re(g\circ\phi_{k, r_k})(y) - Re(g\circ
\phi_{k,r_k})(\theta)}{y- \theta} \right| = +\infty
\end{equation}
for all $\theta\in\mathbb{R}$. The above implies that for all 
$\theta \in \mathbb{R}$:

\begin{equation}
\limsup\limits_{t\rightarrow 0^+} \left| \frac{Re(g\circ\phi_{k, r_k})(\theta + t\alpha) - Re(g\circ
\phi_{k,r_k})(\theta)}{t\alpha} \right| = +\infty
\end{equation}
for any $\alpha\in\mathbb{R}$ with $\dfrac1s\leq |\alpha|\leq s$.

Now consider the function $f: \overline{\Omega} \longrightarrow \mathbb{C}$ which maps 
$({z_i})_{i\in I}$ to $g(z_k)$. We will prove that $f\in S_{k,s}^{(r)}(I)$. Let $\theta\in\mathbb{R}^I$ 
and $v\in\mathbb{R}^I$ such that $|v_j|\leq s$ for all $j\in I$ and $\dfrac1s\leq |v_k|\leq s$. We observe 
that

\begin{equation}
\left| \frac{u_{r}(\theta + tv) - u_{r}(\theta)}{t\norm{v}_\infty} 
\right| = \dfrac{|v_k|}{\norm{v}_\infty}\left| 
\frac{\tilde{u}_{k, r_k}(\theta_k + tv_k) - \tilde{u}_{k,r_k}(\theta_k)}{tv_k} 
\right|,
\end{equation}
does not remain bounded as $t\to 0$, where $u_{r}=Ref\circ\phi_{r}$ and $\tilde{u}_{k,r_k}=Ref\circ\phi_{k,r_k}$. Thus $f\in S_{k,s}^{(r)}(I)$ and $S_{k,s}^{(r)}(I)\neq \emptyset$. 

Let us now show that $S_{k,s}^{(r)}(I)$ is dense in $A(\Omega)$. We remind the reader that

\begin{align}
\phi_r: \overline{D}^I &\longrightarrow \overline{\Omega_{1,r_1}} \times 
\overline{\Omega_{2,r_2}} \times ... = \prod_{i \in I}\overline{\Omega}_{i,r_i} \\
\nonumber (z_i)_{i\in I}& \longmapsto( \phi_{i,r_i}(z_i))_{i\in I}
\end{align}
and that the function

\begin{align}
\phi_r^{-1}: \overline{\Omega_{1,r_1}} \times \overline{\Omega_{2,r_2}} \times ...  &
\longrightarrow \overline{D}^I \\
\nonumber (z_i)_{i\in I}& \longmapsto( \phi_{i,r_i}^{-1}(z_i))_{i\in I}
\end{align}
is the inverse of $\phi_r$.
It easily follows that $\phi_r\in A(D^I)$ and $\phi_r^{-1}\in A(\Omega_{1,r_1}\times
\Omega_{2,r_2}\times...)$. Now, note that $\phi_r^{-1}(\overline{\Omega})=\prod\limits_{i\in I} 
\phi_{i,r_i}^{-1}(\overline{\Omega_i})$ where $\phi_{i,r_i}^{-1}(\overline{\Omega_i})$ is a 
compact set in $\mathbb{C}$ of finite connectivity and consider the functions

\begin{align}
\widetilde{\phi_r}: \phi_r^{-1}(\overline{\Omega}) &\longrightarrow \overline{\Omega} \\
\nonumber z & \longmapsto \phi_r(z)
\end{align}
and 

\begin{align}
\widetilde{\phi^{-1}_r}: \overline{\Omega} &\longrightarrow \phi_r^{-1}(\overline{\Omega}) \\
\nonumber z & \longmapsto \phi_r^{-1}(z)
\end{align}

We observe that $\widetilde{\phi_r^{-1}}$ is the inverse of $\widetilde{\phi_r}$. Also, note that if $h\in A(\Omega)$, then $h\circ\widetilde{\phi_r}\in A(\phi_r^{-1}
(\overline{\Omega}))$. Since $\phi_r^{-1}(\overline{\Omega})$ is a product of domains in 
$\mathbb{C}$ bounded by a finite number of disjoint Jordan curves, then according to the result of \cite{[3]} in the case of finitely many variables, the set 
$U$ of functions which are finite sums of finite products of rational functions of one complex 
variable $z_i$ with poles off $\Omega_i$ is dense in $A(\phi_r^{-1}
(\overline{\Omega}))$. Hence, the set $f+ U\circ \widetilde{\phi_r^{-1}}$ is 
dense in $A(\Omega)$. If $q\in U$, then for every  $t\in\mathbb{R}$, $\theta\in
\mathbb{R}^I$ and 
$v\in\mathbb{R}^I$ with $|v_j|\leq s$ for all $j\in I$ and $\dfrac1s\leq |v_k|\leq s$ we have that:

\begin{equation}
\left| \frac{(q\circ\widetilde{\phi_r^{-1}})(\phi_r(\theta+tv)) - (q\circ
\widetilde{\phi_r^{-1}})(\phi_r(\theta))}{t\norm{v}_\infty} \right| = \dfrac{1}{\norm{v}_
\infty}\left| \frac{q(\theta+tv) - q(\theta)}
{t} \right|.
\end{equation}
remains bounded. This follows from the fact that every rational function of one complex variable 
$z_i$ which appears in $q$ has poles off $\overline{\Omega_i}$. Thus, the function
$p(t)=q(\theta+tv)$ is differentiable at $0$. Therefore, $f+ U\circ 
\widetilde{\phi_r^{-1}}\subset S_{k,s}^{(r)}(I)$, which completes the proof.
\end{proof}

Baire's theorem combined with Lemma \ref{many} and Lemma \ref{dense} imply the following theorem.

\begin{theorem}\label{intersection}
The class
\begin{equation}
S(I) = \bigcap\limits_{r\in R}\bigcap\limits_{k\in I}\bigcap\limits_{s\in\mathbb{N}} S^{(r)}_{k,s}(I)
\end{equation}
is a dense and $G_\delta$ subset of $A(\Omega)$.
\end{theorem}

Consequently, $S\cap iS$ is dense and $G_\delta$ in $A(\Omega)$ and it is a subset of the set of all functions $f\in A(\Omega)$ such that both $Re(f\circ\phi_r)\restriction_{\mathbb{T}^I}$ and $Im(f\circ\phi_r)\restriction_{\mathbb{T}^I}$ are nowhere differentiable for any $r\in R$ and any direction $v\in l^\infty(I)\setminus\{0\}$.
\subsection{Infinitely many variables}

Let $I$ be an infinitely countable set.

We consider again the sets
\begin{gather}
D_{n,k,s}(I)=\{u\in C_{\mathbb{R}}(\mathbb{T}^I):\text{for every } \theta\in\mathbb{R}^I
 \text{ and every direction } \\
v\in\mathbb{R}^I \nonumber \text{ with } |v_j|\leq s \text{ for any } j\in I \text{ and } \dfrac1s\leq |v_k|\leq s \text{ there exists } \\
 y\in\left(\theta-\dfrac{1}{n}v,\theta+\dfrac{1}{n}v\right) \text{ such that } 
\nonumber |u(\theta)-u(y)|>n\Vert y-\theta \Vert_{\infty}\}
\end{gather}
and
\begin{equation}
E_{n,k,s}(I)=\bigcap_{r\in R}\{f\in A(\Omega): Re(f\circ\phi_r)\restriction_{\mathbb{T}^I}\in D_{n,k,s}\}, 
\end{equation}
for $n,s \in\mathbb{N}$ and $k\in I$.

\begin{lemma}\label{infinite}
For every $n,s\in \mathbb{N}$ and $k\in I$ the set $E_{n,k,s}(I)$ is an open subset of $A(\Omega)$.
\end{lemma}

\begin{proof}
We will prove that the set $A(\Omega)\setminus E_{n,k,s}(I)$ is closed. We consider a sequence
$\{f_m\}$ such that every $f_m\in A(\Omega)\setminus E_{n,k,s}(I)$, which converges to a 
function $f\in A(\Omega)$. For each $m$ there exist $r^{(m)}\in R$,
$\theta^{(m)}\in \mathbb{T}^\mathbb{N}$ and a direction $v^{(m)}\in\mathbb{R}^\mathbb{N}$
with $|v^{(m)}_j|\leq s$ for all $j\in I$ and $\dfrac1s\leq |v^{(m)}_k|\leq s$, such that
\begin{equation}
|(u_m\circ\phi_{r^{(m)}})(y)-(u_m\circ\phi_{r^{(m)}})(\theta^{(m)})|\leq n
\Vert{ y-\theta^{(m)}}\Vert_\infty
\end{equation}
for every 
$y\in\left(\theta^{(m)}-\dfrac{1}{n}v^{(m)},\theta^{(m)}+\dfrac{1}{n}v^{(m)}\right)$, where
$u_m=Ref_m$.
But $R$, $\mathbb{T}^\mathbb{N}$ and the set 
$\{v\in\mathbb{R}^\mathbb{N}: |v_j|\leq s,\forall j\in I \text{ and } \dfrac1s \leq |v_k|\leq s\}$ are metrizable and compact by 
Tychonoff's theorem. Thus, they have converging subsequences. Without loss of generality we 
assume that $r^{(m)}\rightarrow r\in R$, 
$\theta^{(m)}\rightarrow\theta\in\mathbb{R}^\mathbb{N}$ and 
$v^{(m)}\rightarrow v$, where $v\in \mathbb{R}^\mathbb{N}$ and $|v_j|\leq s$ for all $j\in I$ and $\dfrac1s\leq |v_k|\leq s$.

If $y\in\left(\theta-\dfrac{1}{n}v,\theta+\dfrac{1}{n}v\right)$, then there exists a $t\in\mathbb{R}$ with 
$-\dfrac{1}{n}<t<\dfrac{1}{n}$ such that:
\begin{equation}
y=\theta+tv=\lim_{m\rightarrow\infty}(\theta^{(m)}+tv^{(m)}).
\end{equation}
Let $i\in I$. Since $R= \prod_{i\in I} \{ 0, ..., k_i- 1 \}$ is a product of finite sets and $r^{(m)} \to r\in\mathbb{R}$, there exists 
$m_0\in\mathbb{N}$, such that 
\begin{equation}
\phi_{i,{r_i}^{(m)}}=\phi_{i,r_i}, \forall m \geq m_0.
\end{equation}

Since $\phi_{i,r_i}$ is continuous, we have that
\begin{equation}
\phi_{i,{r_i}^{(m)}}({\theta_i}^{(m)}+t{v_i}^{(m)})\rightarrow 
\phi_{i,r_i}(y_i),
\quad \phi_{i,{r_i}^{(m)}}({\theta_i}^{(m)})\rightarrow \phi_{i,r_i}
(\theta_i)
\end{equation}
Hence,
\begin{equation}
\phi_{r_{(m)}}(y^{(m)})\rightarrow \phi_r(y), \quad \phi_{r_{(m)}}(\theta^{(m)})
\rightarrow \phi_r(\theta)
\end{equation}
in the product topology. Therefore, applying $(3.5)$, using $(3.9)$ and
letting $m\rightarrow\infty$ we conclude that
\begin{equation}
|(u\circ\phi_r)(y)-(u\circ\phi_r )(\theta)|\leq n\Vert y-\theta \Vert_\infty,
\end{equation}
where $u=Ref$,
since $f_m\rightarrow f$ uniformly. Thus, $f\in A(\Omega)\setminus E_{n,k,s}(I)$ and 
$A(\Omega)\setminus E_{n,k,s}(I)$ is closed.
\end{proof}

\begin{theorem}\label{final}
The class
\begin{equation}
S(I)=\bigcap_{k\in I}\bigcap_{s=1}^{\infty}\bigcap_{n=1}^\infty E_{n,k,s}(I)
\end{equation}
is a dense and $G_\delta$ subset of $A(\Omega)$.
\end{theorem}

\begin{proof}
Lemma \ref{infinite} implies that $S$ is a $G_\delta$. Thus, we only have to prove that $S$ is dense.
Let $F\in A(\Omega)$. From \cite{[3]} there exist a sequence of finite sets 
$\{I_n\}$, $I_n \subset I$, and a sequence 
$\{F_n\}$ in $A(\Omega)$, such that each
$F_n$ depends only on the variables $z_i$, $i\in I_n$ and $F_n\rightarrow F$. 
Let $f_n$ be the restriction of $F_n$ to $\prod_{i\in I_n} \Omega_i$.
From Theorem \ref{intersection} there exists a sequence $g_n\in S(I_n)=\bigcap_{r\in R(I_n)} \bigcap_{k\in I_n}\bigcap_{s=1}^{\infty}S_{k,s}^{(r)}(I_n)$, where $R(I_n)= \prod_{i\in I_n} \{ 0, ..., k_i -1\}$, such that
$\norm{f_n-g_n}_\infty\rightarrow 0$. It is easy to see that the functions 
$G_n:\overline{\Omega}\rightarrow\mathbb{C}$ which map $(z_i)_{i\in I}$ to 
$g_n((z_i)_{i\in I_n})$ respectively belong to $ S(I) = \bigcap_{k\in I}\bigcap_{s=1}^{\infty}\bigcap_{n=1}^{\infty}E_n(I)$ and $G_n\rightarrow F$. 
The proof is complete.
\end{proof}

The analogue of Theorem \ref{curves} follows immediately: if $I$ is a non empty finite set or an 
infinitely countable set, then the set 
\begin{equation}
T(I)=S(I)\cap iS(I)
\end{equation}
is a dense and $G_\delta$ subset of $A(\Omega)$. This gives the non directional differentiability at every boundary point of both real and imaginary part of the function in every 
direction in $l^\infty(I)\setminus\{0\}$.

\begin{remark}
In \cite{[2]} it is claimed that the set of functions $f\in A(D^\mathbb{N})$ for which for every 
$\theta\in\mathbb{R}^\mathbb{N}$ and every direction $v\in\mathbb{R}^\mathbb{N}$ with
$\norm{v}_{\infty}= 1$ and $\dfrac12\leq |v_k|\leq1$ for a fixed $k\in\mathbb{N}$ there 
exists 
$y\in\left(\theta-\dfrac{1}{n}v,\theta+\dfrac{1}{n}v\right)$ such that 
\begin{equation}
|u(\theta)-u(y)|>n\Vert y-\theta \Vert_{\infty}
\end{equation} 
is open in $A(D^\mathbb{N})$. The authors' proof uses the wrong statement that the set
$\{v\in\mathbb{R}^\mathbb{N} \text{ with }
\norm{v}_{\infty}= 1 \text{ and } \dfrac12\leq |v_k|\leq 1\}$ for a fixed $k\in\mathbb{N}$ 
is compact in $\mathbb{R}^\mathbb{N}$ endowed with the product topology. However, their proof 
works just fine for every compact set of directions with non zero $k$-coordinate. According to our Theorem \ref{final} their statement is correct.
\end{remark}

\noindent \textbf{Acknowledgement}: We would like to express our gratidute to
Professor Vassili Nestoridis for introducing us to the problem as well as for his guidance throughout the creation of this paper.

\end{document}